\documentclass[a4paper,12pt]{article}
\usepackage[numbers,sort&compress]{natbib}
\usepackage{amsmath,amsthm,amsfonts,amssymb,enumerate,calc,graphicx,iwona,float}
\usepackage[colorlinks=true,citecolor=black,linkcolor=black,urlcolor=blue]{hyperref}
\usepackage[tmargin=30mm,bmargin=30mm,rmargin=30mm,lmargin=30mm]{geometry}
\usepackage[noabbrev,capitalise]{cleveref}
\newcommand{\arXiv}[1]{arXiv:\,\href{http://arxiv.org/abs/#1}{#1}}

\newcommand{\msn}[1]{MR:\,\href{http://www.ams.org/mathscinet-getitem?mr=MR#1}{#1}}

\newcommand{\doi}[1]{doi:\,\href{http://dx.doi.org/#1}{#1}}

\renewcommand{\baselinestretch}{1.2}
\usepackage[hang]{footmisc}

\renewcommand{\thefootnote}{\fnsymbol{footnote}}	
\allowdisplaybreaks
\newcommand\DateFootnote{
\begingroup
\renewcommand\thefootnote{}
\footnote{\today}
\setcounter{footnote}{0}
\vspace*{-3ex}
\endgroup}
\makeatletter
\renewcommand\section{\@startsection {section}{1}{\z@}{-3ex \@plus -1ex \@minus -.2ex}{2ex \@plus.2ex}{\normalfont\large\bfseries}}
\renewcommand\subsection{\@startsection{subsection}{2}{\z@}{-2.5ex\@plus -1ex \@minus -.2ex}{1.5ex \@plus .2ex}{\normalfont\normalsize\bfseries}}
\renewcommand\subsubsection{\@startsection{subsubsection}{3}{\z@}{-2ex\@plus -1ex \@minus -.2ex}{1ex \@plus .2ex}{\normalfont\normalsize\bfseries}}
 \renewcommand\paragraph{\@startsection{paragraph}{4}{\z@}{1.5ex \@plus.5ex \@minus.2ex}{-1em}{\normalfont\normalsize\bfseries}}
\renewcommand\subparagraph{\@startsection{subparagraph}{5}{\parindent}  {1.5ex \@plus.5ex \@minus .2ex}  {-1em} {\normalfont\normalsize\bfseries}}
\makeatother
\renewcommand{\thefootnote}{\fnsymbol{footnote}}

\newcommand{\ceil}[1]{\lceil{#1}\rceil}
\newcommand{\floor}[1]{\lfloor{#1}\rfloor}

\renewcommand{\geq}{\geqslant}
\renewcommand{\leq}{\leqslant}

\DeclareMathOperator{\td}{td}

\theoremstyle{plain}
\newtheorem{theorem}{Theorem}
\newtheorem{lemma}[theorem]{Lemma}

\newtheorem{proposition}[theorem]{Proposition}

\begin{document}

{\Large\bfseries\boldmath\scshape Colourings with Bounded Monochromatic\\ Components in Graphs of given Circumference}

\medskip
Bojan Mohar\footnotemark[3] \quad
Bruce Reed\footnotemark[2] \quad
David R. Wood\footnotemark[4]

\DateFootnote

\footnotetext[3]{Department of Mathematics, Simon Fraser University, Burnaby, Canada, (\texttt{mohar@sfu.ca}). On leave from IMFM, Ljubljana, Slovenia. Research supported by an NSERC Discovery Grant, CRC program, and in part by ARRS, Research Program P1-0297.}

\footnotetext[2]{CNRS, Projet COATI, I3S (CNRS and UNS) UMR7271 and INRIA, Sophia Antipolis, France,  (\texttt{reed@i3s.unice.fr}). Instituto Nacional de Matem\'atica Pura e Aplicada (IMPA), Brasil. Visiting Research Professor, ERATO Kawarabayashi Large Graph Project, Japan. }

\footnotetext[4]{School of Mathematical Sciences, Monash University, Melbourne, Australia, \texttt{david.wood@monash.edu}. Supported by the Australian Research Council.}

\emph{Abstract.} We prove that every graph with circumference at most $k$ is $O(\log k)$-colourable such that every monochromatic component has size at most $k$. The $O(\log k)$ bound on the number of colours is best possible, even in the setting of colourings with bounded monochromatic degree.

%

\hrulefill

In a vertex-coloured graph, a \emph{monochromatic component} is a connected component of the subgraph induced by all the vertices of one colour.  As a relaxation of proper colouring, recent work has focused on graph colourings with monochromatic components of bounded size \cite{ADOV-JCTB03,EJ14,LO15a,BK14,EO16,AUW15} or bounded monochromatic degree \citep{OOW16,Archdeacon1987,EKKOS2014,HS2006,DKMR2014,BK13,BIMR11,MR3528002,MR3300755,MR3095439,MR2876359,MR2489421}.

The \emph{circumference} of a graph $G$ is the length of the longest cycle if $G$ contains a cycle,  and is 2 if $G$ is a forest. This paper studies colourings of graphs of given circumference with monochromatic components of bounded size. Our primary goal is to minimize the number of colours, while reducing the order of the monochromatic components is a secondary objective.

Let $g(k)$ be the minimum integer $c$ for which there exists an integer $d$ such that every graph with circumference at most $k$ has a $c$-colouring in which every monochromatic component has order at most $d$. Our main result is that $g(k)=\Theta(\log k)$. First we prove the upper bound.

\begin{theorem}
\label{Circumference}
For every integer $k \geq 2$, every graph $G$ with circumference at most $k$ is $(3 \log_{2} k)$-colourable such that every monochromatic component has order at most $k$.
\end{theorem}

This result is implied by the following lemma with $C=\emptyset$. A \emph{clique} is a set of pairwise adjacent vertices.

\begin{lemma}
For every integer $k \geq 2$, for every graph $G$ with circumference at most $k$ and for every pre-coloured clique $C$ of size at most $2$ in $G$,
there is a $\floor{3\log_2 k}$-colouring of $G$ such that every monochromatic component has order at most $k$ and
every monochromatic component that intersects $C$ is contained in $C$.
\end{lemma}

\begin{proof}
We proceed by  induction on $k+|V(G)|$. The result is trivial if $|V(G)|\leq 2$. Now assume $|V(G)|\geq 3$. 

First suppose that $k=2$. Then $G$ is a forest, which is properly 2-colourable. If $|C|\le 1$ or $|C|=2$ and two colours are used on $C$, we obtain the desired colouring (with $2< \floor{3\log_2 k}$ colours). Otherwise, $|C|=2$ with the same colour on the vertices in $C$. Contract the edge $C$ and 2-colour the resulting forest by induction, to obtain the desired colouring of $G$. Now assume that $k \geq 3$.

Suppose that $G$ is not $3$-connected. Then $G$ has a minimal separation $(G_1,G_2)$ with $S := V(G_1 \cap G_2)$ of size at most 2. If $|S|=2$, then add the edge on $S$ if the edge is not already present. Consider both $G_1$ and $G_2$ to contain this edge.
Observe that since the separation is minimal, there is a path in each $G_j$ ($j=1,2$) between the two vertices of $S$. Therefore, adding the edge does not increase the circumference of $G$. Also note that any valid colouring of the augmented graph will be valid for the original graph. 
Since $C$ is a clique, we may assume that $C \subseteq V(G_1)$. By induction, there is a $\floor{3\log_2 k}$-colouring of $G_1$, with $C$ precoloured, such that every monochromatic component of $G_1$ has order at most $k$
and every monochromatic component of $G_1$ that intersects $C$ is contained in $C$. This colours $S$.
By induction, there is a $\floor{3\log_2 k}$-colouring of $G_2$, with $S$ precoloured, such that every monochromatic component of $G_2$  has order at most $k$
and every monochromatic component of $G_2$  that intersects $S$ is contained in $S$. By combining the two colourings, every monochromatic component of $G$ has order at most $k$
and every monochromatic component of $G$ that intersects $C$ is contained in $C$, as required.
Now assume that $G$ is 3-connected.

Every 3-connected graph contains a cycle of length at least 4. Thus $k\geq 4$. 

If $G$ contains no cycle of length $k$, then apply the induction hypothesis for $k-1$; thus we may assume that $G$ contains a cycle $Q$ of length $k$. 
Let $\mathcal{A}$ be the set of cycles in $G$ of length  at least $\ceil{\tfrac12 (k-5)}$. 
Suppose that a cycle $A \in \mathcal{A}$ is disjoint from $Q$. Since $G$ is 3-connected, there are three disjoint paths between $A$ and $Q$. It follows that $G$ contains three cycles with total length at least 
$2(|A|+|Q|+3) > 3k$. Thus $G$ contains a cycle of length greater than $k$, which is a contradiction.
Hence, every cycle in $\mathcal{A}$ intersects $Q$.

Let $S:=V(Q)\cup C$. As shown above, $G':=G-S$ contains no cycle of length at least $\ceil{\frac{1}{2}(k-5)}$. 
Then $G'$ has circumference at most $\ceil{\frac{1}{2}(k-7)}$, which is at most $\floor{\tfrac12 k}$, which is at least 2. By induction (with no precoloured vertices), there is a $\floor{3\log_{2} \floor{\frac12 k}}$-colouring of $G'$ such that every monochromatic component of $G'$ has order at most $\floor{\frac12 k}$. Use a set of colours for $G'$ disjoint from the (at most two) preassigned colours for $C$. Use one new colour for $S\setminus C$, which has size at most $k$. In total, there are at most $\floor{3\log_2 \floor{\frac12 k} }+3 \leq  \floor{3\log_2 k}$ colours. Every monochromatic component of $G$ has order at most $k$, 
and every monochromatic component of $G$ that intersects $C$ is contained in $C$.
%
\end{proof}

Note that if $h$ is the function defined by the recurrence, 
\begin{align*}
h(k)  :=
\begin{cases}
2 & \text{ if }k=2\\
5 & \text{ if } 3\leq k\leq 9\\
h(\ceil{\frac12 (k-7)})+3 & \text{ if } k\geq 10,
\end{cases}
\end{align*}
then $\floor{3\log_2 k}$ can be replaced by $h(k)$ in \cref{Circumference}. 

\bigskip
We now show that the $O(\log k)$ bound in \cref{Circumference} is within a constant factor of optimal even in the setting of colourings of bounded monochromatic degree. The following result is implicit in \citep{OOW16}. We include the proof for completeness.

\begin{proposition}
\label{LowerBound}
For any integers $k,d\geq 1$ there is a graph $G_{k,d}$ with circumference at most $2^k$,
such that every $k$-colouring of $G_{k,d}$ contains a vertex of monochromatic degree at least $d$.
\end{proposition}

\begin{proof}
We proceed by induction on $k\geq 1$ with $d$ fixed (and thus write $G_k$ instead of $G_{k,d}$), and with the additional property that $G_k$ contains no path of order $2^{k+1}$. For the base case, $k=1$, let $G_1$ be the star $K_{1,d}$, which has circumference 2 and no path of order 4. Every 1-colouring of $G_1$ contains a vertex of monochromatic degree $d$. Now assume that $k\geq 2$ and there is a graph $G_{k-1}$ with circumference $2^{k-1}$ and no path of order $2^k$, such that every $(k-1)$-colouring of $G_{k-1}$ contains a vertex of monochromatic degree at least $d$. Let $G_k$ be obtained from $d$ copies of $G_{k-1}$ by adding one new dominant vertex $v$.

If $C$ is a cycle in $G_k$ with length at least $2^k+1$, then $C$ is contained in one copy of $G_{k-1}$ plus $v$, and thus $G_{k-1}$ contains a path of order $2^k$, which is a contradiction. Thus $G_k$ has circumference at most $2^k$.

If $G_k$ contains a path $P$ of order $2^{k+1}$, then $v$ is in $P$, otherwise $P$ is contained in some copy of $G_{k-1}$. Hence $P-v$ includes a path component of order $\ceil{ \frac12 (2^{k+1}-1) } = 2^k$ contained in a copy of $G_{k-1}$, which is a contradiction. Hence  $G_k$ contains no path of order $2^{k+1}$.

Finally, consider a $k$-colouring of $G_k$. Say $v$ is blue. If every copy of $G_{k-1}$ contains a blue vertex, then $v$ has monochromatic degree $d$, and we are done. Otherwise, some copy of $G_{k-1}$ contains no blue vertex, in which case $G_{k-1}$ is $(k-1)$-coloured, and thus $G$ contains a monochromatic vertex of degree at least $d$.
\end{proof}

Let $f(k)$ be the minimum integer $c$ for which there exists an integer $d$ such that every graph with circumference at most $k$ has a $c$-colouring in which every monochromatic component has maximum degree at most $d$. In the language of \citet{OOW16}, $f(k)$ is the defective chromatic number of the class of graphs with circumference at most $k$. Obviously, bounded size implies bounded degree, so $f(k)\leq g(k)$. \cref{Circumference,LowerBound} imply that 
\begin{equation}
\label{LowerUpperBounds}
\floor{\log_2 k}+1 \leq f(k) \leq g(k)\leq \floor{3\log_2 k}.
\end{equation}

We conclude this paper by placing our results in the context of a conjecture of \citet{OOW16}. The \emph{closure} of a rooted tree $T$ is the graph obtained from $T$ by adding an edge between each ancestor and descendant. The \emph{tree-depth} of a graph $H$, denoted $\td(H)$, is the minimum depth of a rooted tree for which $H$ is a subgraph of the closure of $T$, where the \emph{depth} of a rooted tree $T$ is the maximum number of vertices in a root-to-leaf path. For a graph $H$, let $f(H)$ be the minimum integer $c$ such that there exists an integer $d$ such that every $H$-minor-free graph has a $c$-colouring in which every monochromatic component has maximum degree at most $d$. \citet{OOW16} proved that 
$f(H)\geq \td(H)-1$ for every connected graph $H$, and conjectured that
\begin{equation}
\label{conj}
f(H)=\td(H)-1.
\end{equation}
A graph has circumference $k$ if and only if it contains no $C_{k+1}$ minor; thus $f(k)=f(C_{k+1})$. 
It is easily seen that 
$$\td(C_{k+1})=1+\ceil{ \log _{2}(k+1)}=2+\floor{\log_2 k}.$$
Thus the lower bound $f(H)\geq \td(H)-1$, in the case of cycles, is equivalent to the lower bound on $f(k)$  
in \eqref{LowerUpperBounds}. And conjecture \eqref{conj}, in the case of cycles, asserts that equality holds. That is, 
$$f(k)=f(C_{k+1})=\td(C_{k+1})-1=\ceil{ \log _{2}(k+1)}.$$
Hence \cref{Circumference}, which proves that $f(k)\leq \floor{3\log_2k}$, is within a factor 3 of conjecture \eqref{conj} for excluded cycles. The best previous upper bound was linear in $k$.

We obtain similar results for graph classes excluding a fixed path, which were identified by \citet{OOW16} as a key case for which their bounds on $f$ were far apart. Let $P_k$ be the path on $k$ vertices. Then $\td(P_k)=\ceil{ \log _{2}(k+1)}$; see \citep{Sparsity}. Of course, a graph contains a $P_k$ minor if and only if it contains a $P_k$ subgraph. Thus conjecture \eqref{conj}, in the case of paths, asserts that $$f(P_{k+1})=\td(P_{k+1})-1=\ceil{ \log _{2}(k+2)}-1.$$ Every graph with no $P_{k+1}$-minor has circumference at most $k$. Thus \cref{Circumference} implies that $f(P_{k+1})\leq \floor{3 \log_2 k}$, which is within a factor of 3 of   conjecture \eqref{conj} for excluded paths. The best previous upper bound was linear in $k$.



\paragraph{Acknowledgement.} This research was completed at the Australasian Conference on Combinatorial Mathematical and Combinatorial Computing (40ACCMCC) held at The University of Newcastle, December 2016. Thanks to the conference organisers.


\def\soft#1{\leavevmode\setbox0=\hbox{h}\dimen7=\ht0\advance \dimen7
  by-1ex\relax\if t#1\relax\rlap{\raise.6\dimen7
  \hbox{\kern.3ex\char'47}}#1\relax\else\if T#1\relax
  \rlap{\raise.5\dimen7\hbox{\kern1.3ex\char'47}}#1\relax \else\if
  d#1\relax\rlap{\raise.5\dimen7\hbox{\kern.9ex \char'47}}#1\relax\else\if
  D#1\relax\rlap{\raise.5\dimen7 \hbox{\kern1.4ex\char'47}}#1\relax\else\if
  l#1\relax \rlap{\raise.5\dimen7\hbox{\kern.4ex\char'47}}#1\relax \else\if
  L#1\relax\rlap{\raise.5\dimen7\hbox{\kern.7ex
  \char'47}}#1\relax\else\message{accent \string\soft \space #1 not
  defined!}#1\relax\fi\fi\fi\fi\fi\fi}

\end{document}